\documentclass[leqno, english]{amsart}
\usepackage{amsfonts}
\usepackage{amsthm}
\usepackage{amsmath}
\usepackage{amscd}
\usepackage[latin2]{inputenc}
\usepackage{t1enc}
\usepackage[mathscr]{eucal}
\usepackage{indentfirst}
\usepackage{graphicx}
\usepackage{graphics}
\usepackage{pict2e}
\usepackage{epic}
\usepackage{bbm}
\usepackage[left=1in,right=1in,top=1in,bottom=1in]{geometry}
\usepackage{epstopdf} 

\usepackage{hyperref}
\hypersetup{
	colorlinks = true,
    linkcolor = blue,
    citecolor = blue,
    urlcolor = black,
}

\theoremstyle{plain}
\newtheorem{thm}{Theorem}
\newtheorem{lem}[thm]{Lemma}

\newtheorem{claim}[thm]{Claim}

\theoremstyle{definition}

\newtheorem*{remark}{Remark}
\newtheorem{problem}[thm]{Problem}
\newtheorem*{example}{Example}

\newcommand{\expect}{\mathbb E}

\newcommand{\Z}{\mathbb Z}

\newcommand{\vsum}[3]{\sum_{#2 \ni #1} #3(#2)}

\newcommand{\cC}{\mathcal{C} }

\newcommand{\cH}{\mathcal{H} }

\newcommand{\cK}{\mathcal{K} }

\newcommand{\pr}{\mathbb{P}}

\newcommand{\beq}[1]{\begin{equation}\label{#1}}
\newcommand{\enq}[0]{\end{equation}}

\newcommand{\nin}[0]{\noindent}

\newcommand{\sub}[0]{\subseteq}
\newcommand{\sm}[0]{\setminus}

\newcommand{\ov}[0]{\overline}

\newcommand{\pp}{\mathbf{p}}
\newcommand{\half}{\mathbf{1/2}}

\newcommand{\0}[0]{\emptyset}

\newcommand{\C}[2]{{{#1}\choose{{#2}}}}

\newcommand{\gd}[0]{\delta }
\newcommand{\gD}[0]{\Delta }

\newcommand{\gl}[0]{\lambda }

\newcommand{\eps}[0]{\varepsilon }

\newcommand{\cond}{\mid}

\newcommand{\prh}[1][]{\pr_h}

\begin{document}

\title{Reverse discrepancy and almost zero-sum stars}

\author{Quentin Dubroff}
\email{qcd2@math.rutgers.edu}
\address{Department of Mathematics, Rutgers University \\
Hill Center for the Mathematical Sciences \\
110 Frelinghuysen Rd.\\
Piscataway, NJ 08854-8019, USA}

\begin{abstract}
For $f$ chosen from the $\{-1,1\}$-valued function on the edges of a hypergraph $\mathcal{H} = (V,E)$ with $\sum_{e \in E} f(e) = 0$, how large can one make $\min_{v\in V} |\sum_{e \ni v} f(e)|$? This question may be viewed as a reverse version of the hypergraph discrepancy problem or as a relaxation of the zero-sum Ramsey problem for stars. We prove exact results when $\mathcal{H}$ is a complete or equipartite hypergraph.
\end{abstract}

\maketitle

\section{Introduction}

The work in this paper was originally motivated by a notion introduced by Balogh and Smyth \cite{BS}. They define the \emph{unbalancedness} of a $\{-1,1\}$-valued function $f$ on the edges of a hypergraph $\cH= (V,E)$ as
\[X(f) = \min_{v\in V} \,\Bigl|\sum_{v\in A} f(A)\Bigr|.\]
For a fixed hypergraph and $\alpha \in [-1, 1]$, they consider the problem of determining 
\[X(\mathcal{H},\alpha) = \max \{X(f) : \sum_{A\in E} f(A) = \alpha E(\mathcal{H})\}.\]
This problem may be viewed as a reverse of the classical hypergraph discrepancy problem (see e.g.\ \cite{Mat})---instead of determining $\min_f \max_{v\in V} |\sum_{v \in A} f(A)|$, we ask for $\max_f \min_{v\in V} |\sum_{v \in A} f(A)|$. The definition of reverse discrepancy is slightly at odds with the convention in hypergraph discrepancy theory to study imbalances on edges for functions on the vertices of a hypergraph. The definitions and results of this paper could be stated equivalently in this way (using hypergraph duality), but for the problems under consideration here, it seems more natural to reverse the roles of vertices and edges.

Balogh and Smyth prove an upper bound on $X(\mathcal{H}, \alpha)$ when $\cH$ is regular and equipartite i.e.\ when $\cH$ is a regular subgraph of the complete $r$-partite hypergraph with vertex classes of the same size, say $n$. Letting $\mathcal{H}$ be regular of degree $Dn^{r-1}$, they show
\beq{BSres}
X(\cH,\alpha) \leq \sqrt{\frac{D + \alpha^2 D^2(r-1)}{r}}n^{r-1}.
\enq
They remark that this is most likely not best possible. However, \eqref{BSres} gives the correct $\Theta(1/\sqrt{r})$ dependence on $r$ for $\alpha = 0$ (with $D$ a fixed constant) and is essentially optimal for large $r$ when $\alpha \neq 0$. We confine ourselves to the case $\alpha = 0$ and define $X(\cH) = X(\cH,0)$. Our first result is a generalization and best possible improvement of \eqref{BSres} in this case. We give an upper bound on $X(\cdot)$ over the class of $r$-uniform $r$-equipartite hypergraphs with fixed numbers of vertices and edges. This upper bound is achieved often, depending on some divisibility constraints. As a special case, we show that if $\mathcal{K}$ is the complete $r$-partite $r$-uniform hypergraph with vertex classes of size $n$, then
\beq{complete}
X(\mathcal{K}) \leq \binom{r-1}{\lfloor r/2 \rfloor}\Big(\frac{n}{2}\Big)^{r-1}.
\enq
Equality holds when $n$ is divisible by four. The full statement of our result involves description of the maxima and is deferred to Theorem~\ref{partiteThm} in Section~\ref{partite}.

The reverse discrepancy problem, like hypergraph discrepancy, has antecedents in hypergraph Ramsey theory. Caro and Yuster \cite{CY} study a notion related to $X(\cH,\alpha)$ for functions on graphs in connection with zero-sum Ramsey theory.
They prove that if a function $f:E(K_n) \to [-1,1]$ on the edges of the complete graph satisfies $|\sum_{e\in E} f(e)| \leq 2(n-1)$, then for every graph $H$ with maximum degreee $\gD$, there is a copy $H'$ in $K_n$ such that $|\sum_{e\in H'} f(e)| \leq 2\gD$. Mohr, Pardey, and Rautenbach \cite{MPR} improve upon a special case of this result, showing (in our notation) that $X(K_n) \leq n/2 - 1$, which is easily seen to be best possible. Using a technique developed to prove Theorem~\ref{partiteThm}, we extend this result to complete hypergraphs of every uniformity. 
\begin{thm}\label{compthm}
Let $\cK_n^{r}$ be the complete $r$-uniform hypergraph on $n$-vertices. Then with $k = \lceil n/2 \rceil$,
\[X(\cK_n^{r}) \leq  \frac{\sum_{\ell=0}^r |\lceil r/2 \rceil -\ell|\binom{n-k}{\ell}\binom{k}{r-\ell}}{k}\]
\end{thm}
\nin A natural example, described in Section~\ref{compsec}, shows this bound is exact in the context of functions with codomain $[-1,1]$ when $n$ is even.

In obtaining these results, we exploit one of the main differences between discrepancy and reverse discrepancy. The reverse discrepancy problem seems to be changed very little by working with the larger class of functions having codomain the interval $[-1,1]$. In fact, the upper bounds presented in this paper hold for this broader class of functions.

\section{Partite Hypergraphs}\label{partite}

Here, $\{-1,1\}^r$ is endowed with the product measure $\mu = \mu_{\pp}$ with $\pp = (p_1,p_2,\ldots, p_r)$ (so $\mu(\{x_i = 1\}) = p_i$), and we write $\mu_{\bf{1/2}}$ when all $p_i=1/2$. The expectation of $f$ with respect to $\mu$ is denoted $\expect_\mu[f]$. Define
\[\overline{X}_\mu(f) = \min_i \Big( \min\Big\{ \expect_\mu [f \cond x_i = 1],-\expect_\mu [f \cond x_i = -1] \Big\}\Big ),\]
which is an analogue of $X(f)$ for functions $f: \{-1,1\}^r \rightarrow [-1,1]$.
We call such a function a \emph{semi-threshold function} if it is of the form
\[
f(x) = \begin{cases}
1 & |x| > k,\\
\beta & |x| = k,\\
0 & -k < |x| < k,\\
-\beta & |x| = -k,\\
-1 & |x| < -k,

\end{cases}
\]
where $|x| = \sum_i x_i$. In this section, we show that maximizers of $X(f)$ on equipartite hypergraphs are essentially semi-threshold functions.

\begin{thm}\label{partiteThm}
Let $\cH = (V,E)$ be an $r$-uniform $r$-equipartite hypergraph with $rn$ vertices, and let $g$ be the semi-threshold function satisfying $\expect_{\mu_\half}[|g|] \leq |E|/n^r$ with $\expect_{\mu_\half}[|g|]$ as large as possible. Then
\[X(\cH) \leq \ov{X}_{\mu_{\half}}(g)n^{r-1}.\]
\end{thm}

\nin Note the maximum of $\ov{X}_{\mu_\half} (\cdot)$ over the class of semi-threshold functions $f$ with $\expect_{\mu_{\half}} |f| \leq \gl$ is attained either at the unique semi-threshold function with $\expect_{\mu_{\half}} |f| = \gl$ or at the semi-threshold function with threshold $k = \beta =0$ (the latter holding when $r$ is even and $\gl > 1 - \C{r}{r/2}/2^r$). If $f$ is a semi-threshold function (with parameters $k$ and $\beta$) and $\mu = \mu_\half$, then $\expect_\mu [f \cond x_i = 1] = -\expect_\mu [f \cond x_i = -1] = \expect_\mu [f x_i]$, and thus 
\[\ov{X}_{\mu_\half}(f) = \expect_\mu [\sum_i f x_i]/r = \left[2(1-\beta)\C{r-1}{(k+r)/2} + 2\beta \C{r-1}{(k+r-2)/2}\right]\big/2^r.\] Therefore Theorem~\ref{partiteThm} gives an explicit upper bound on $X(\cH)$ for $r$-uniform $r$-equipartite hypergraphs. See Section~\ref{equality} for a discussion of tight cases.

The proof of Theorem~\ref{partiteThm} proceeds by reducing the problem to a highly symmetric case in which we can apply the main lemma of this section:

\begin{lem}\label{main}
Let $f: \{-1,1\}^r \rightarrow [-1,1]$ be such that $\expect_{\mu_{\pp}}[f] = 0$, and let $g$ be the semi-threshold function satisfying $\expect_{\mu_{\half}} [|g|] \leq \expect_{\mu_{\pp}}[|f|]$ with $\expect_{\mu_\half}[|g|]$ as large as possible. Then
\[\overline{X}_{\mu_{\pp}}(f) \leq \overline{X}_{\mu_{\half}}(g).\]
Furthermore, the inequality is strict if $\expect_{\mu_{\pp}}[|f|] = 1$ and some $p_i$ is not $1/2$.
\end{lem}

\begin{remark}
Lemma~\ref{main} is a generalization of \cite[Theorem 2.33]{OD}, which says that the majority function is the boolean function maximizing the sum of the level one fourier coefficients.
\end{remark}

\begin{proof} If $\pp \neq (1/2,\ldots, 1/2)$, we first employ a shifting procedure jointly on the function and measure which ends at the uniform measure without reducing $\overline{X}$. Formally, let $f: \{-1,1\}^r \rightarrow [-1,1]$ and $\mu = \mu_{\pp}$ be such that $\expect_{\mu}[f] = 0$ and $\expect_{\mu}[|f|] = \lambda$. Suppose without loss of generality that $p_1 < 1/2$. We show there is $f'$ such that with $\pp' = (1/2, p_2,\ldots, p_r)$ and $\mu' = \mu_{\pp'}$, $\overline{X}_{\mu'}(f') \geq \overline{X}_\mu(f)$, $\expect_{\mu'}[f'] = 0$, and $\expect_{\mu'}[|f'|] \leq \lambda$.

Let $x^1$ denote the element of $\{-1,1\}^r$ obtained by switching the first coordinate of $x$, and define
\beq{shift}
f'(x) = 
\begin{cases}
2 p_1 f(x) + (1-2p_1)f(x^1) & x_1 = 1,\\
f(x) & x_1 = -1;
\end{cases}
\enq
It is easy to see that $f' : \{-1,1\}^r \rightarrow[-1,1]$ and, crucially, 
\[ \mu'(x^1)f'(x^1) + \mu'(x) f'(x) =  \mu(x^1) f(x^1) + \mu(x)f(x)\]
for every $x \in \{-1,1\}^r$.
Moreover, 
$$\mu'(x^1)|f'(x^1)| + \mu'(x) |f'(x)| \leq \mu(x^1)|f(x^1)| + \mu(x)|f(x)|,$$
by the triangle inequality.
This implies that 
$\expect_{\mu'}[f']  = \expect_\mu [f] = 0$ and $\expect_{\mu'}[|f'|] \leq \expect_\mu[|f|]$. Similarly for $i\neq 1$, $\expect_{\mu'} [f' \cond x_i = 1] = \expect_{\mu} [f \cond x_i = 1]$ and $\expect_{\mu'} [f' \cond x_i = -1] = \expect_{\mu} [f \cond x_i = -1]$.
Furthermore, as $f'(x) = f(x)$ when $x_1 = -1$,
\[\expect_{\mu'} \big[f' \cond x_1 = -1\big] = \expect_{\mu} \big[f \cond x_1 = -1\big].
\]
Lastly, as $\mu'(\{x_1 = 1\}) = \mu'(\{x_1 = -1\}) = 1/2$ and $\expect_{\mu'} [f'] = 0$,
\[\expect_{\mu'} \big[f' \cond x_1 = 1 \big] = -\expect_{\mu'} \big[f' \mid x_1 = -1\big].\]
These equations show that $\overline{X}_{\mu'}(f') \geq \overline{X}_\mu(f)$ (and as already noted, $\expect_{\mu'}[|f'|] \leq \expect_\mu[|f|]$). We continue this process until $p_i = 1/2$ for all $i$.

We may therefore assume that $\mu = \mu_\half$. In this case, $\expect_\mu [f | x_i = 1] = \expect_\mu [f x_i]$, as $\expect_\mu f = 0$, and we may assume this expectation is non-negative. Furthermore, the maximum of 
\[
    \sum_i \expect_\mu \big[f x_i\big]  =  \expect_\mu \big[f \sum_i  x_i \big],
\]
subject to $\expect_\mu [|f|] \leq \lambda$ is attained at a semi-threshold function. Therefore
\[\overline{X}(f) = \min_i \expect_\mu[f x_i] \leq \frac{1}{r} \expect_\mu \big[f \sum_i  x_i \big] \leq  \frac{1}{r} \expect_\mu \big[g \sum_i  x_i \big],\]
where $g$ is a semi-threshold function satisfying $\expect_{\mu_{\half}} [|g|] \leq \lambda$ and $\expect_{\mu_{\half}} [|g|]$ is as large as possible.

If $\lambda = 1$, the upper bound is concise, and the maximizer is unique: We have
\begin{equation}\label{OD}
\expect_{\mu_\half} \big[f \sum_i  x_i \big] \leq 2r \binom{r-1}{\lfloor r/2 \rfloor}/2^r
\end{equation}
with equality only if $f = \text{sgn}(\sum_i x_i)$ when $\sum_i x_i \neq 0$. Moreover, this maximum cannot be achieved if some $p_i \neq 1/2$: Suppose $f'$ is the result of shifting one or more $p_i$ to $1/2$. The averaging in \eqref{shift} guarantees an $x$ with $f'(x) \in (-1,1)$. This already shows strict inequality in (\ref{OD}) when $r$ is odd because $\expect_\mu \big[f \sum_i  x_i \big]$ has a unique maximizer when $r$ is odd. For $r$ even, more care is needed. We may assume that $f'$ was obtained from $f$ after shifting $p_1 < 1/2$ to $1/2$ and $p_i = 1/2$ for $i> 1$. If equality holds in (\ref{OD}), then $f'(x) = 1$ when $\sum_i x_i > 0$. Due to the definition of $f'$ in \eqref{shift}, this can only happen if $f(x) = 1$ for every $x$ such that $x_1 = -1$ and $\sum_i x_i \geq 0$. However, in this case,
\[
\expect [f' \cond x_1 = -1] \geq 0,
\]
which shows that $\ov{X}(f') = 0$.
\end{proof}

\begin{proof}[Proof of Theorem~\ref{partiteThm}]

Let $\mathcal{K}$ be the complete $r$-uniform $r$-equipartite hypergraph on the same vertex partition $C_1\cup  C_2 \cup \cdots \cup C_r$ as $\cH$. Consider $f:E(\mathcal{H}) \rightarrow [-1,1]$ with $f(E(\mathcal{H})) = 0$ and define $C_i^1 = \{v\in C_i : \sum_{A \ni v}f(A)) \geq 0\}$, $C_i^{-1} = C_i \setminus C_i^1$. This allows for a natural map $\phi$ from the set of edges of $\mathcal{K}$ to the set of $r$-tuples with entries in $\{-1,1\}$, namely $\phi(A)_i = 1$ if and only if $A \cap C_i \sub C_i^1$. We say that two edges $A$ and $B$ in $E(\mathcal{K})$ are in the same class if $\phi(A) = \phi(B)$ and denote this relationship $A\sim B$.

\begin{claim}\label{redlem}
There exists $f' : E(\mathcal{K}) \rightarrow [-1,1]$ such that 
\begin{gather}
    \label{reduced} f'(A) = f'(B) \text{ if } A \sim B;\\
    \label{partmono} X(f') \geq X(f);\\
    \label{bal} f'(E(\mathcal{K}))=0.
\end{gather}
\end{claim}

\begin{proof}
We think of $f: E(\cK) \to [-1,1]$ by setting $f(B) = 0$ if $B \not\in E(\mathcal{H})$.
For $A\in E(\cK)$ let $c_A = |\{B : B \sim A\}|$, and set 
\[f'(A) = \frac{1}{c_A}\sum_{B \sim A} f(B).\]
To see \eqref{partmono}, fix $v$ in, say, $C_1^1$. Since 
\[\sum_{A \in \cC} f'(A) = \sum_{A \in \cC} f(A)\]
for any class $\cC$, we have
\[f'(\{A \ni v\}:= \sum_{A \ni v} f'(A) = \frac{1}{|C_1^1|} \sum_{A \cap C_1^1 \neq \0} f'(A) = \frac{1}{|C_1^1|} \sum_{A \cap C_1^1 \neq \0} f(A) \geq \min_{u \in C_1^1} f(\{A \ni u\}) \geq X(f),\]
where the last inequality holds by definition of $C_1^1$. Similar reasoning shows that for any $v \in C_i^{-1}$, $f'(\{A \ni v\}) \leq -X(f)$, so we have shown \eqref{partmono}. It is clear that $f'$ also satisfies \eqref{reduced} and \eqref{bal}.
\end{proof}

It remains to bound $X(f')$ for $f'$ satisfying the conclusions of Claim~\ref{redlem}. Property \eqref{reduced} allows us to define $h:\{-1,1\}^r \to [-1,1]$ by $h(\phi(A)) = f'(A)$ (which defines $h$ on all of $\{-1,1\}^r$ since $\phi$ must be surjective from the condition $f(E(\cH)) = 0$ and the free assumption $X(f) > 0$). Property (\ref{bal}) then implies $\mathbb{E}_{\mu}h = 0$ when $\mu = \mu_{\pp}$ is the product measure on $\{-1,1\}^r$ with $p_i = |C_i^1|/|C_i|$. Moreover, if $v \in C_i^1$, then by (\ref{reduced}),
\[f'(\{A \ni v\}) = \sum_{x \in \{-1,1\}^r, \atop x_i = 1} h(x) \prod_{j \neq i} |C_j^{x_j}| = \expect_{\mu_\pp} [h \cond x_i = 1]\prod_{j \neq i} |C_j|\]
(and $f'(\{A \ni v\}) = \expect_{\mu_\pp} [h | x_i = -1]\prod_{j \neq i} |C_j|$  if $v\in C_i^{-1}$). Therefore by (\ref{partmono}),
\[X(f) \leq X(f') = \ov{X}(h) n^{r-1}.\]
Lastly, the averaging of Claim~\ref{redlem} gives
\[\expect_{\mu_\pp} |h| \leq \frac{|E(\mathcal{H})|}{\prod_{j=1}^r |C_j|}.\]
Applying Lemma~\ref{main} to $h$ completes the proof.
\end{proof}

\subsection{Cases of equality}\label{equality}

The upper bound of Theorem~\ref{partiteThm} is best possible whenever the averaging procedure of Claim~\ref{redlem} outputs a semi-threshold function attaining the maximum in Lemma~\ref{main}.
To illustrate, we describe the simple construction showing equality in \eqref{complete} in a case which avoids divisibility issues.

\begin{example}
Let $\cK$ be the complete $r$-uniform $r$-partite hypergraph with vertex classes of size $n$, in the case that $r$ is odd and $n$ is even. We define $f: E(\cK) \to \{-1,1\}$ achieving the maximum imbalance. In each vertex class, fix half of the vertices in each class to `+' and half to `$-$'. Set $f(A)=1$ if the number of `+' vertices in $A$ outweighs the number of `$-$' and $f(A)=-1$ otherwise. It is easy to check that $f$ achieves the upper bound (\ref{complete}).
\end{example}

There is a large class of $r$-uniform $r$-equipartite hypergraphs achieving the upper bound of Theorem~\ref{partiteThm} when four divides $n$. To generalize the example above, let the weight $w_A$ of an edge $A$ be the number of `+' vertices minus the number of `$-$' vertices. Fix an integer $k>0$ and for $A\in E(\cK)$, set $f(A) = 1$ if $w_A \geq k$, $f(A) = -1$ if $w_A \leq -k$, and $f(A) = 0$ otherwise. Let $\cH$ be the subhypergraph of $\cK$ obtained by removing all edges $A$ with $f(A) = 0$. Then $f : E(\cH) \to \{-1,1\}$ and it is easy to check in this case that all inequalities in the proof of Theorem~\ref{partiteThm} are equalities. More generally, one can partition the set of edges of weight $k$ into $\{E_i\}_{i=1}^\ell$ such that all parts have the same number of edges and in each $E_i$, every `+' vertex is in the same number of edges and every `$-$' vertex is in the same number of edges. Similarly, one can partition the edges of weight $-k$ into $\{F_i\}_{i=1}^\ell$ with parts of the same size. Then for $I,J \subseteq [\ell]$ with $|I| = |J|$, let $E_I = \cup_{i \in I} E_i$ and $F_J = \cup_{j\in J} F_j$. The hypergraph $\cH \sm (E_I \cup F_J)$ with $f$ defined above is also a case of equality for Theorem~\ref{partiteThm}. 

\section{The complete hypergraph}\label{compsec}

Let $\cK_n^{r} = (V,E)$ be the complete $r$-uniform hypergraph with $|V| = n$. We aim to determine the maximum over $f:E\rightarrow [-1,1]$ satisfying $\sum_A f(A) = 0$ of 
\[X(f) = \min_v |\sum_{A \ni v} f(A)|.\]
When $n$ is even, a natural candidate for the maximizer is the following. Partition the vertex set into two equal parts $V_1$ and $V_2$. Define the function $f$ by $f(A) = 1$ if $|A\cap V_1| > r/2$, $f(A) = -1$ if $|A\cap V_1| < r/2$, and $f(A) = 0$ if $|A\cap V_1| = r/2$. In this section, we prove an upper bound that is exact for this example.

Our analysis proceeds by showing that for any maximizer $f$, there is a partition of the vertex set $V = P \cup N$ such that $f(A)$ is determined by $|A\cap P|$. In this symmetric setting, we employ a type of shifting to show that the maximum is only achieved on an (almost) equipartition. 

\subsection{Reduction to the symmetric case}\label{compred}

\begin{lem}\label{compave}
Given $f: E(\cK_n^r) \rightarrow [-1,1]$ satisfying $\sum_A f(A) = 0$, there are a function $f' : E(\cK_n^r) \rightarrow [-1,1]$ and a partition $V = P \cup N$ such that 
\begin{gather}
    \label{itm:ave} \sum_A f'(A) = 0;\\
    \label{itm:mono} X(f') \geq X(f);\\
    \label{itm:det} |A \cap P| = |B \cap P| \implies f'(A) = f'(B).
\end{gather}
\end{lem}
\begin{proof}
Set $P = \{v : \sum_{A \ni v} f(A) > 0\}$ and $N = V \setminus P$. Define $E_j = \{A : |A \cap P| = j\}$, and for $A \in E_j$, set $f'(A) = \sum_{B \in E_j} f(B) / |E_j|$. It is clear that $f'$ satisfies properties (\ref{itm:ave}) and (\ref{itm:det}). Observe that
\[\sum_{v\in P} \sum_{A\ni v} f'(A) = \sum_{A} f'(A) |A \cap P| = \sum_{A} f(A) |A \cap P| = \sum_{v\in P} \sum_{A\ni v} f(A).\]
The inner sum on the left is the same for all $v\in P$, so for each such $v$,
\[\sum_{A \ni v} f'(A) = \frac{1}{|P|} \sum_{w\in P} \sum_{A \ni w} f(A) \geq \min_{w\in P} \sum_{A \ni w} f(A).\]
Similarly for any vertex $v \in N$, $\sum_{A \ni v} f'(A) \leq \max_{w\in N} \sum_{A \ni w} f(A)$, which establishes \eqref{itm:mono}.
\end{proof}

\subsection{Shifting}

In the remainder of the section we will assume \eqref{itm:det} holds (with $f$ in place of $f'$) for the partition $V = P \cup N$ where $\vsum{v}{A}{f} > 0$ if $v \in P$ and $\vsum{v}{A}{f} < 0$ if $v \in N$. Set $|P| = n-k$ where $0 < k < n$, so $|E_i| = \C{n-k}{i}\C{k}{r-i}$ (recalling $E_i := \{A : |A\cap P| = i\}$). It is clear in this case that if $v \in P$ and $w \in N$, then
\begin{equation}\label{eq:tot}
    |P|\vsum{v}{A}{f} + |N|\vsum{w}{A}{f} = 0.
\end{equation} 
\begin{claim}\label{range}
Let $f$ maximize $X(\cdot)$ over all functions mapping $E(\cK_n^r)$ to $[-1,1]$ and suppose $f$ satisfies (\ref{itm:det}). Then $f(A) \leq f(B)$ if $A\in E_i,B \in E_{i+1}$ and there is at most one $i$ such that $|E_i| \neq 0$ and $|f(A)| \neq 1$ for $A\in E_i$.
\end{claim}
\begin{proof}
It is enough to show that if $i \in \{0,\ldots, r-1\}$ such that $|E_i| |E_{i+1}| \neq 0$ and $-1 < f(A)$ for $A\in E_i$ and $f(B) < 1$ for $B\in E_{i+1}$, then we may modify $f$ to increase $X(f)$. Choose $\eps > 0$ so that $f(A) - \eps \geq -1$ and $f(B) + \delta \leq 1$, where $\delta = \eps|E_i|/|E_{i+1}|$. Set $f'(A) = f(A) - \eps$ for $A\in E_i$, $f'(B) = f(B) + \delta$ for $B\in E_{i+1}$, and $f'(C) = f(C)$ otherwise. Then
\[\sum_{v\in P}\vsum{v}{A}{f'} - \sum_{v\in P}\vsum{v}{A}{f} = \sum_A (f'(A) - f(A))|A\cap P| = (i+1) \gd |E_{i+1}| - i\eps|E_i| = \eps|E_i| > 0,\]
so the common value of $\vsum{v}{A}{f'}$ for $v\in P$ is greater than $\vsum{v}{A}{f}$. Similarly (or by (\ref{eq:tot})), $\vsum{w}{A}{f'} < \vsum{w}{A}{f}$ when $w \in N$, which shows that $X(f') > X(f)$.
\end{proof}
\begin{proof}[Proof of Theorem~\ref{compthm}]
Let $f$ maximize $X(\cdot)$ over all functions mapping $E(\cK_n^r)$ to $[-1,1]$. Lemma~\ref{compave} allows us to assume that $V = P \cup N$ such that \eqref{itm:det} holds, so we may assume the conclusion of Claim~\ref{range}. Recall $|P| = n-k$ and $E_i =\{A : |A \cap P| = i\}$.

Let $m = m(k)$ be the largest integer $i$ such that $f(A) < 1$ for $A \in E_i$ and $|E_i| \neq 0$. As $\sum_A f(A) = 0$, we must have $m \geq 0$. Under the assumptions on $f$ and again using $\sum_A f(A) = 0$, we have
\[f(A) = \frac{\sum_{\ell = 0}^{m-1}\binom{n-k}{\ell}\binom{k}{r-\ell} - \sum_{\ell = m+1}^{r}\binom{n-k}{\ell}\binom{k}{r-\ell}}{\binom{n-k}{m}\binom{k}{r-m}}\]
for $A \in E_m$. Substituting this value of $f$ in the expression for $\vsum{v}{A}{f}$, we find $X(f) = \chi(k)$, where
\[\chi(k) = \frac{\sum_{\ell=0}^r |m-\ell|\binom{n-k}{\ell}\binom{k}{r-\ell}}{\max( n-k, k)}.\]
We are finished if we can show $\chi(k)$ is maximized when $k = \lfloor n/2\rfloor$. $\chi(k)$ is symmetric about $n/2$, so we assume without loss of generality that $k \leq n/2 - 1$, and we will show 
$$F(k) = \frac{\sum_{\ell=0}^r |m-\ell|\binom{n-k}{\ell}\binom{k}{r-\ell}}{n-k}$$
is increasing in $k$. It is straightforward to check that either $m(k) = m(k+1)$ or $m(k) = m(k+1) + 1$. 
For $s \in \{0,1,\ldots, r\}$, consider 
$$G(k,s) = \frac{\sum_{\ell=0}^r |s-\ell|\binom{n-k}{\ell}\binom{k}{r-\ell}}{n-k}$$
Using the binomial coefficient identity $\C{n+1}{k} = \C{n}{k} + \C{n}{k-1}$ and collecting terms, we have
\begin{align*}
    G(k+1,s) - G(k,s) &= \frac{\sum_{\ell=0}^r |s-\ell|\binom{n-k-1}{\ell}\binom{k+1}{r-\ell}}{n-k-1} - \frac{\sum_{\ell=0}^r |s-\ell|\binom{n-k}{\ell}\binom{k}{r-\ell}}{n-k}\\
    &= \frac{\sum_{\ell=0}^r |s-\ell| (n-k) \big(\binom{n-k}{\ell} - \binom{n-k-1}{\ell -1}\big)\big(\binom{k}{r-\ell} + \binom{k}{r - \ell -1}\big)}{(n-k)(n-k-1)}\\
    &\phantom{=} -\frac{\sum_{\ell = 0}^r |s-\ell|(n-k-1)\binom{n-k}{\ell}\binom{k}{r-\ell}}{(n-k)(n-k-1)}\\
    &= \frac{\sum_{\ell=0}^r |s-\ell| (1-\ell)\binom{n-k}{\ell}\binom{k}{r-\ell} + |s-\ell|(\ell + 1)\binom{n-k}{\ell + 1}\binom{k}{r-\ell-1}}{(n-k)(n-k-1)}\\
    &= s \frac{\sum_{\ell = 0}^s \binom{n-k}{\ell}\binom{k}{r - \ell} - \sum_{\ell = s+1}^r \binom{n-k}{\ell}\binom{k}{r - \ell}}{(n-k)(n-k-1)}.
\end{align*}
When $s = m(k)$, the last expression is non-negative by the definition of $m$, which shows that $F(k)$ increases if $m(k) = m(k+1)$. When $m(k) = m(k+1) + 1$, we have 
\begin{align*}
F(k+1) - F(k) &= G(k+1, m(k+1)) - G(k, m(k+1)) - \frac{\sum_{\ell = 0}^{m(k)-1} \binom{n-k}{\ell}\binom{k}{r - \ell} - \sum_{\ell = m(k)}^r \binom{n-k}{\ell}\binom{k}{r - \ell}}{(n-k)}\\
&=\frac{\big[(n-k-1)-m(k+1)\big]\big[\sum_{\ell = 0}^{m(k)-1} -\binom{n-k}{\ell}\binom{k}{r - \ell} +    \sum_{\ell = m(k)}^r \binom{n-k}{\ell}\binom{k}{r - \ell}\big]}{(n-k)(n-k-1)},
\end{align*}
where both expressions within parentheses are non-negative by the definition of $m$.
\end{proof}

\section{Concluding remarks and problems}

We have given a tight upper bound on $X(\cH)$ when $\cH$ is an $r$-uniform $r$-equipartite hypergraph with $e$ edges. It would be nice to achieve such an upper bound for arbitrary $r$-uniform $r$-partite hypergraphs. It seems that for general $r$-partite graphs, the maximum of $X(\cH)$ might be achieved at functions which look like weighted semi-threshold functions, where the weights account for the different sizes of the blocks of the vertex partition. 

Theorem~\ref{compthm} gives a tight upper bound on $X(\cK_n^r)$ where $\cK_n^r$ is the complete $r$-uniform hypergraph. It would be interesting to extend this to arbitrary $r$-uniform hypergraphs.
\begin{problem}\label{maximize}
Determine the maximum (or, more likely, give bounds that are often tight) of $X(\cdot)$ over the class of $r$-uniform hypergraphs with $n$ vertices and $e$ edges.
\end{problem}
While the results of this paper strongly suggest what the maximizers of Problem~\ref{maximize} might look like, determining the minimum of $X(\cdot)$ over the class of hypergraphs with a fixed number of edges and vertices seems difficult, even in the case of graphs. It seems possible that $X(G) > d/2 - o(d)$ for every $d$-regular graph $G$. This is supported by results on friendly/internal partitions \cite{BL,BSc,Parts,LinLou,Steib}, which study the existence of a (non-trivial) vertex partition $V = V_0 \cup V_1$ such that every $v\in V_i$ has more neighbors in $V_i$ than in $V_{1-i}$. Given such a partition of the vertices of $G$, a reasonable first step to showing $X(G) > d/2 - o(d)$ is to assign all edges within $V_0$ weight $-1$ and all within $V_1$ weight $1$. It is unclear how to assign weights to the rest of the edges, but perhaps a partition with stronger properties (e.g.\ as in Problem 4.2 of \cite{LinLou}) would help. More generally, the following seems interesting:
\begin{problem}
Find asymptotically tight lower bounds for $X(\cdot)$ over the class of $d$-regular $r$-uniform hypergraphs.
\end{problem}

\section{Acknowledgements}

Thanks to Jeff Kahn and Bhargav Narayanan for helpful conversations and many comments on the manuscript. Additional thanks to Jeff Kahn for suggesting the problem which gave rise to this paper. This work was supported in part by Simons Foundation grant 332622.


\begin{thebibliography}{99}

\bibitem{BS}
J.\ Balogh and C.\ Smyth, \emph{On the variance of Shannon products of graphs}, Discrete Applied Mathematics \textbf{156}, 110-118, 2008.

\bibitem{BL}
A.\ Ban and N.\ Linial, \emph{Internal Partitions of Regular Graphs} J.\ Graph Theory \textbf{83}, 5-18, 2016.

\bibitem{BSc}
B.\ Bollob\'as and A.D.\ Scott, \emph{Problems and results on judicious partitions}, Random Struct. Alg. \textbf{21}, 414-430, 2002.

\bibitem{CY}
Y.\ Caro and R.\ Yuster, \emph{On zero-sum and almost zero-sum subgraphs over $\Z$}, Graphs and Combinatorics \textbf{32}, 49-63, 2016.

\bibitem{Parts}
A.\ Ferber, M.\ Kwan, B.\ Narayanan, A.\ Sah, M.\ Sawhney, \emph{Friendly bisections of random graphs}, \href{https://arxiv.org/abs/2105.13337v2}{arXiv:2105.13337v2}.

\bibitem{LinLou}
N.\ Linial and S.\ Louis, \emph{Asymptotically almost every $2r$-regular graph has an internal partition}, Graphs and Combinatorics \textbf{36}, 41-50, 2020.

\bibitem{Mat}
J.\ Matou\v{s}ek, \textit{Geometric Discrepancy, an Illustrated Guide}, Springer, New York, 1999.

\bibitem{MPR}
E.\ Mohr, J.\ Pardey, and D.\ Rautenbach, \emph{Zero-sum copies of spanning forests in zero-sum complete graphs}, \href{http://arxiv.org/abs/2101.11233v1}{arXiv:2101.11233v1}.


\bibitem{OD}
R.\ O'Donnell, \emph{Analysis of Boolean Functions}, Cambridge University Press, 2014.

\bibitem{Steib}
M.\ Stiebitz, \emph{Decomposing graphs under degree constraints}, J.\ Graph Theory \textbf{23}, 321-324, 1996.

\end{thebibliography}
\end{document}